\documentclass[
final
, nomarks
]{dmtcs-episciences}

\usepackage[utf8]{inputenc}

\usepackage{amsmath,amssymb,amsthm}
\usepackage{arydshln}

\newtheorem{thm}{Theorem}[section]
\newtheorem{lem}[thm]{Lemma}

\theoremstyle{definition}
\newtheorem{defn}[thm]{Definition}

\author{Michael D. Barrus\affiliationmark{1}}
\title[Weakly threshold graphs]{Weakly threshold graphs}
\affiliation{Department of Mathematics, University of Rhode Island, USA}

\received{2017-9-17}

\accepted{2018-5-4}

\begin{document}
\publicationdetails{20}{2018}{1}{15}{3968}
\maketitle
\begin{abstract}
We define a \emph{weakly threshold sequence} to be a degree sequence $d=(d_1,\dots,d_n)$ of a graph having the property that $\sum_{i \leq k} d_i \geq k(k-1)+\sum_{i > k} \min\{k,d_i\} - 1$ for all positive $k \leq \max\{i:d_i \geq i-1\}$. The \emph{weakly threshold graphs} are the realizations of the weakly threshold sequences. The weakly threshold graphs properly include the threshold graphs and satisfy pleasing extensions of many properties of threshold graphs. We demonstrate a majorization property of weakly threshold sequences and an iterative construction algorithm for weakly threshold graphs, as well as a forbidden induced subgraph characterization. We conclude by exactly enumerating weakly threshold sequences and graphs.
\end{abstract}

\section{Introduction} \label{sec: intro}

The threshold graphs are a remarkable and well-studied class of graphs. As explained in the monograph devoted to them by Mahadev and Peled~\cite{MahadevPeled95}, these graphs have been independently rediscovered in diverse contexts, and they have a large number of equivalent characterizations. For example, Chv\'{a}tal and Hammer~\cite{ChvatalHammer73,ChvatalHammer77} defined threshold graphs as those graphs whose vertices can be labeled with nonnegative numerical values so that a set of vertices is an independent set if and only if the values of the included vertices sum to at most some predetermined value (the ``threshold''). Other characterizations of threshold graphs have dealt with characteristics ranging from construction algorithms to forbidden induced subgraphs to eigenvalues of the Laplacian matrix; see~\cite{MahadevPeled95} for a broad introduction.

One characterization of threshold graphs, due to Hammer, Ibaraki, and Simeone, concerns their degree sequences, which we call \emph{threshold sequences}. In this and all other results in this paper, we assume that degree sequences are indexed with their terms ordered from largest to smallest. Given a degree sequence $d=(d_1,\dots,d_n)$, we further define $m(d) = \max\{i: d_i \geq i-1\}$.

\begin{thm}[\cite{HammerIbarakiSimeone78}] \label{thm: threshold deg seq char}
Let $d=(d_1,\dots,d_n)$ be the degree sequence of a graph $G$. The graph $G$ is a threshold graph if and only if \[\sum_{i=1}^k d_i = k(k-1) + \sum_{i>k} \min\{k,d_i\}\] for all $k \in \{1,\dots,m(d)\}$.
\end{thm}

This theorem bears a strong resemblance to a well known theorem of Erd\H{o}s and Gallai characterizing graphic sequences. (The version stated here uses an improvement due to Hammer, Ibaraki, and Simeone.)
\begin{thm}[\cite{ErdosGallai60, HammerIbarakiSimeone78, HammerIbarakiSimeone81}] \label{thm: Erdos Gallai}
A sequence $d=(d_1,\dots,d_n)$ of nonnegative integers, with even sum and terms in nonincreasing order, is the degree sequence of a simple graph if and only if 
\begin{equation}\label{eq: Erdos Gallai}\sum_{i=1}^k d_i \leq k(k-1) + \sum_{i>k} \min\{k,d_i\}\end{equation} for all $k \in \{1,\dots,m(d)\}$.
\end{thm}
Thus Theorem~\ref{thm: threshold deg seq char} shows that threshold sequences are in one sense extremal examples among all degree sequences.

The Erd\H{o}s--Gallai inequalities of Theorem~\ref{thm: Erdos Gallai} are derived from the observation that the number of edges joining vertices with large degree to vertices of low degree cannot exceed the capacity of the low-degree vertices to accommodate these edges. As we might expect, in order for threshold sequences to satisfy these inequalities with equality, the adjacencies in a threshold graph are rigidly determined. In fact, one of the remarkable properties of threshold sequences is that each such sequence has exactly one labeled realization, and threshold sequences are the only degree sequences with this property~\cite{ChvatalHammer73, FulkersonHoffmanMcAndrew65}.

Stated another way, in a threshold graph the presence or absence of an edge between two vertices is uniquely determined by the degrees of those two vertices. In a recent paper~\cite{ForcedEdges}, the author characterized the circumstances under which an edge (or non-edge) is forced to appear in all realizations of a degree sequence. The answer can be stated in terms of the quantities
\begin{equation} \label{eq: EG differences}
\Delta_k(d) = k(k-1) + \sum_{i>k} \min\{k,d_i\} - \sum_{i=1}^k d_i
\end{equation}
for $1 \leq k \leq m(d)$, which we call the \emph{Erd\H{o}s--Gallai differences} of $d$. By Theorem~\ref{thm: Erdos Gallai} the Erd\H{o}s--Gallai differences are all nonnegative for any degree sequence. As shown in~\cite{ForcedEdges}, in order for an adjacency relationship to be constant among all labeled realizations of a degree sequence, it is necessary that an Erd\H{o}s--Gallai difference be at most 1.

Because of Theorem~\ref{thm: threshold deg seq char}, threshold sequences are precisely those degree sequences where all of the first $m(d)$ Erd\H{o}s--Gallai differences are 0. It is perhaps natural to wonder, though, what properties of threshold graphs may continue to hold in a more general form if this condition is relaxed somewhat. In light of the significance of Erd\H{o}s--Gallai differences of 1, at least in the degree sequence problem of~\cite{ForcedEdges}, we make a definition.
\begin{defn}
A degree sequence $d$ is a \emph{weakly threshold sequence} if for all $k \in \{1,\dots,m(d)\}$ we have $\Delta_k(d) \leq 1$. If such is the case, then every realization of $d$ is called a \emph{weakly threshold graph}.
\end{defn}
Since the four-vertex path is a weakly threshold graph but not a threshold graph, the class of weakly threshold graphs properly contains the class of threshold graphs.

In this paper, we review several characterizations of threshold graphs and show that for most of them, a more general property holds for weakly threshold graphs. In Section~\ref{sec: majorization} we establish some preliminary results on Erd\H{o}s--Gallai differences and show that, as for threshold graphs and threshold sequences, the weakly threshold graphs are split graphs, and the weakly threshold sequences have nearly symmetric Ferrers diagrams and appear at the top of the majorization order on degree sequences. In Section~\ref{sec: iter const} we examine iterative constructions of threshold and weakly threshold graphs. In Section~\ref{sec: forb subgr} we show that the weakly threshold graphs form a hereditary graph class and characterize them in terms of forbidden induced subgraphs; we see that weakly threshold graphs form a notable subclass of both the interval graphs and their complements. In Section~\ref{sec: enum} we enumerate the weakly threshold sequences and graphs and compare these numbers to those of the threshold graphs.

Throughout the paper, we will use $K_n$, $P_n$, and $C_n$, respectively, to denote the complete graph, the path, and the cycle with $n$ vertices. We use $V(G)$ to denote the vertex set of a graph $G$. The \emph{open neighborhood} of a vertex $v$ is the set of vertices adjacent to $v$; the \emph{closed neighborhood} of $v$ is the union of $\{v\}$ and the open neighborhood of $v$. Other terms and notation will be defined as they are encountered.

\section{Preliminaries and majorization} \label{sec: majorization}

In this section we focus on weakly threshold sequences, showing that they satisfy approximate versions of the Ferrers diagram symmetry and majorization properties of threshold sequences, which we describe below. Along the way we will also show that every weakly threshold graph is a split graph.

In discussing graph degree sequences, it has often proved useful to associate with a list $d$ of nonnegative integers its \emph{corrected Ferrers diagram} $C(d)$ (see, for example, the monograph~\cite{MahadevPeled95}, from which we adapt our notation and presentation). Assuming that $d=(d_1,\dots,d_n)$ and that the terms of $d$ are nonincreasing, we define $C(d)$ to be the $n\times n$ matrix with entries drawn from $\{0,1,\star\}$ such that the entries on the main diagonal all equal $\star$, and for each $i \in \{1,\dots,n\}$, the leftmost $d_i$ entries not on the main diagonal are equal to 1, with the remaining entries in the row each equaling 0.

Recall now our definition of $m(d)$, the \emph{corrected Durfee number of $d$}, from the previous section: \[m(d) = \max\{i:d_i \geq i-1\}.\] Pictorially, $m(d)$ represents the side length (measured in entries) of the largest square containing no 0 that occupies the top left corner of $C(d)$. (This square is called the \emph{corrected Durfee square} of $C(d)$.) As an example, in Figure~\ref{fig: C(11110) and C(2211)} we exhibit $C(s)$ and $C(s')$, where $s=(1,1,1,1,0)$ and $s'=(2,2,1,1)$; we see that $m(s) = m(s')=2$.
\begin{figure}
\centering
$\displaystyle
\begin{bmatrix}
\star & 1 & 0 & 0 & 0\\
1 & \star & 0 & 0 & 0\\
1 & 0 & \star & 0 & 0\\
1 & 0 & 0 & \star & 0\\
0 & 0 & 0 & 0 & \star
\end{bmatrix}
\qquad \qquad
\begin{bmatrix}
\star & 1 & 1 & 0\\
1 & \star & 1 & 0\\
1 & 0 & \star & 0\\
1 & 0 & 0 & \star
\end{bmatrix}
$
\caption{The corrected Ferrers diagrams of $(1,1,1,1,0)$ and $(2,2,1,1)$}
\label{fig: C(11110) and C(2211)}
\end{figure}

Given two lists $p=(p_1,\dots,p_j)$ and $q=(q_1,\dots,q_k)$ of nonnegative integers, we say that \emph{$p$ majorizes $q$}, and we write $p \succeq q$, if $\sum_{i=1}^j p_i = \sum_{i=1}^k q_i$ and if for each positive integer $l$, $\sum_{i=1}^l p_i \geq \sum_{i=1}^l q_i$ (where undefined sequence terms are assumed to be 0). 

It is well known that the partitions of any fixed nonnegative integer form a poset under the relation $\succeq$. Furthermore, if $p$ and $q$ are lists of positive integers (ignoring any 0's) that have the same sum, $p$ is a graph degree sequence, and $p \succeq q$, then $q$ is a degree sequence of a graph as well.

Threshold sequences have characterizations in terms of corrected Ferrers diagrams and majorization, as the following theorem shows. We will see shortly that relaxed versions of these statements hold for weakly threshold sequences.

\begin{thm}[see~{\cite[Theorem 3.2.2]{MahadevPeled95}}]\label{thm: deg seq props of threshold graphs}
Let $d$ be a degree sequence. The following are equivalent.
\begin{enumerate}
\item[\textup{(i)}] The sequence $d$ is a threshold sequence.
\item[\textup{(ii)}] The corrected Ferrers diagram $C(d)$ is a symmetric matrix.
\item[\textup{(iii)}] If $e$ is a degree sequence and $e \succeq d$, then $d = e$.
\end{enumerate}
\end{thm}

In order to describe the corrected Ferrers diagrams of weakly threshold sequences, we first give a pictorial interpretation of the Erd\H{o}s--Gallai differences.

\begin{lem}\label{lem: EG diff via diagram}
Let $d$ be a degree sequence. For all $k \in \{1,\dots,m(d)\}$, the Erd\H{o}s--Gallai difference $\Delta_k(d)$ equals $B_k(d)-R_k(d)$, where $B_k(d)$ is the number of 1's in the first $k$ columns of $C(d)$ that lie below the diagonal of stars, and $R_k(d)$ is the number of 1's in the first $k$ rows of $C(d)$ lying to the right of diagonal of stars.
\end{lem}
\begin{proof}
Fix $k \in \{1,\dots,m(d)\}$, and observe that $R_k(d)$ equals $\sum_{i=1}^k d_i - k(k-1)/2$. Further note that $B_k(d)$ equals $k(k-1)/2 + \sum_{i>k} \min\{k,d_i\}$. Subtracting $R_k(d)$ from $B_k(d)$ yields $\Delta_k(d)$.
\end{proof}

Lemma~\ref{lem: EG diff via diagram}, together with Theorem~\ref{thm: threshold deg seq char}, links statements (i) and (ii) in Theorem~\ref{thm: deg seq props of threshold graphs} when we observe that each $1$ in $C(d)$ lies in one of the first $m(d)$ rows or one of the first $m(d)$ columns of the diagram. Note that a degree sequence $d$ is a threshold sequence if and only if each star in $C(d)$ has an equal number of 1's below it and to the right of it.

Lemma~\ref{lem: EG diff via diagram} provides us with a similar statement about $C(d)$ when $d$ is a weakly threshold sequence; in this case, the numbers of stars 1's and to the right of each star can differ by at most 1, so $C(d)$ is ``almost symmetric.'' Futhermore, the instances where the numbers do differ for a given star are constrained by the facts that $0 \leq \Delta_k(d) \leq 1$ for all $k \in \{1,\dots,m(d)\}$; for instance, if the numbers do differ for two stars among the first $m(d)$ stars in $C(d)$, and they do not differ for any stars between these two, then it follows from Lemma~\ref{lem: EG diff via diagram} that one of these two stars must have one more 1 below it than to the right of it, and the other star must have the opposite situation.

Before discussing an analogue for statement (iii) in Theorem~\ref{thm: deg seq props of threshold graphs}, we mention another class of graphs with a degree sequence characterization. A graph is \emph{split} if its vertex set can be partitioned into (possibly empty) sets, where one is a clique and the other is an independent set. As shown in~\cite{HammerSimeone81}, a degree sequence $d$ is the degree sequence of a split graph if and only if $\Delta_m(d)=0$, where $m=m(d)$. It follows from Theorem~\ref{thm: threshold deg seq char} that threshold graphs are split graphs. We now consider weakly threshold graphs.

\begin{lem} \label{lem: Delta_m is even}
If $d$ is a degree sequence and $m=m(d)$, then $\Delta_m(d)$ is an even number.
\end{lem}
\begin{proof}
Since each $1$ in $C(d)$ lies in one of the first $m(d)$ rows or one of the first $m(d)$ columns of the diagram, and the sum of the terms of $d$ is an even number (the sum is twice the number of edges), it follows from Lemma~\ref{lem: EG diff via diagram} that $\Delta_m(d)$ is also even.
\end{proof}

\begin{thm} \label{thm: WT graphs are split}
Every weakly threshold graph is a split graph.
\end{thm}
\begin{proof}
Let $d$ be the degree sequence of a weakly threshold graph $G$, and let $m=m(d)$. By definition, $\Delta_m(d) \leq 1$, so Lemma~\ref{lem: Delta_m is even} implies that $\Delta_m(d)=0$, which in turn implies that $G$ is a split graph.
\end{proof}

We now discuss the majorization order on degree sequences. Statement (iii) of Theorem implies that in the poset of degree sequences ordered by majorization, the threshold sequences are the maximal elements. Our next result shows that weakly threshold sequences, though they include degree sequences that are not maximal in this poset, do form an upward-closed subset of the poset. We first require some preliminary ideas.

A \emph{unit transformation} is an operation on a degree sequence $a=(a_1,\dots,a_n)$ that subtracts 1 from $a_i$ and adds 1 to $a_j$ for indices $i,j \in \{1,\dots,n\}$ such that $a_i \geq a_j+2$. Pictorially, if $a'$ is the resulting sequence (after reordering terms into descending order), then $C(a')$ is obtained by ``moving'' a 1 from one row of $C(a)$ down to the end of the nonzero entries in a lower row (replacing a 0 in that row and being replaced by a 0 in the original row). A well known result states that if $a \succeq b$, then $b$ can be obtained from $a$ after a finite sequence of unit transformations.

\begin{lem}\label{lem: EG diffs preserved upwards}
If the degree sequence $d$ can be obtained by a unit transformation on a degree sequence $e$, and if $m'=\min\{m(d),m(e)\}$, then for each $k \in \{1,\dots,m'\}$ we have $\Delta_k(e) \leq \Delta_k(d)$.
\end{lem}
\begin{proof}
Define $R_k(d)$ and $B_k(d)$ as in the statement of Lemma~\ref{lem: EG diff via diagram}. Since $d$ is obtained through a unit transformation on $e$, it follows that $R_k(d) \leq R_k(e) \leq B_k(e) \leq B_k(d)$; by Lemma~\ref{lem: EG diff via diagram}, we see that $\Delta_k(e) \leq \Delta_k(d)$.
\end{proof}

\begin{lem}\label{lem: if m changes}
Let $d$ and $e$ be degree sequences, and let $m=m(e)$ and $m'=m(d)$. Suppose that $e \succeq d$ and that $d$ is obtained by performing a single unit transformation on $e$. If $m' < m$, then $m=m'+1$ and $\Delta_{m}(e) = \Delta_{m'}(d)$.
\end{lem}
\begin{proof}
Suppose that $d$ and $e$ are as described, with $m' < m$, and suppose that the unit transformation on $e$ that produces $d$ reduces $e_i$ by 1 and increases $e_j$ by 1. Now $e_m -1 \leq d_{m}< m-1 \leq e_m$, so $i=m$ and $e_m = m-1$ and $d_m = m-2$, while $d_l = e_l$ for $1 \leq l \leq m-1$. We then find that $d_{m-1} \geq d_m = m-2$, so $m-1 \leq m' < m$, which implies $m' = m-1$. When $l > m$, we have $e_l \leq e_{m+1} < m$, so $\min\{m,e_l\} = e_l$. Likewise, when $l>m'$ we have $\min\{m',d_l\} = d_l$. Then
\begin{align*}
\Delta_m(e) - \Delta_{m'}(d) &= m(m-1) - m'(m'-1) + \sum_{l > m} \min\{m,e_l\} - \sum_{l > m'} \min\{m',d_l\} - \sum_{l=1}^m e_l + \sum_{l=1}^{m'} d_l\\
&= 2(m-1) + e_j - d_j - d_m - e_m\\
&= 0. \qedhere
\end{align*}
\end{proof}

\begin{thm}\label{thm: WT sequences are upward closed}
If $d$ is a weakly threshold sequence and $e$ is a degree sequence such that $e \succeq d$, then $e$ is also a weakly threshold sequence.
\end{thm}
\begin{proof}
It suffices to prove the result in the case that $d$ is obtained via a single unit transformation on $e$; suppose that this is the case. If $m(e) \leq m(d)$, then the result follows inductively from Lemma~\ref{lem: EG diffs preserved upwards}. If instead $m(e)>m(d)$ then applying Lemmas~\ref{lem: EG diffs preserved upwards} and~\ref{lem: if m changes} inductively we find that each of $\Delta_1(e), \dots, \Delta_{m(e)}(e)$ is equal to one of $\Delta_1(d), \dots, \Delta_{m(d)}(d)$, and hence $e$ is a weakly threshold sequence.
\end{proof}

\section{Iterative construction} \label{sec: iter const}

Threshold graphs have a characterization via a construction algorithm. A \emph{dominating vertex} is a vertex that is adjacent to all other vertices in the graph. An \emph{isolated vertex} is a vertex that is adjacent to none of the other vertices.

\begin{thm}[see~{{\cite[Theorem 1.2.4]{MahadevPeled95}}}] \label{thm: threshold dom/iso}
A graph $G$ is a threshold graph if and only if $G$ can be obtained by beginning with a single vertex and iteratively adding either a dominating vertex or an isolated vertex.
\end{thm}

In this section we show that weakly threshold graphs can be obtained from small initial graphs by repeatedly adding vertices; to generate all weakly threshold graphs we slightly relax the conditions on the adjacencies and non-adjacencies required of the added vertices.

Given a graph $G$ and a vertex $v$ of $G$, we say that $v$ is \emph{weakly dominating} if $v$ is adjacent to every other vertex of $G$ except for a single vertex that has minimum degree in $G$. The vertex $v$ is instead \emph{weakly isolated} if $v$ has no neighbors except for a single vertex that has maximum degree in $G$. A \emph{semi-joined $P_4$} in $G$ is a collection $P$ of 4 vertices that induce a subgraph isomorphic to $P_4$, in which the path midpoints are adjacent to every vertex not in $P$, and the path endpoints are adjacent to no vertex not in $P$. (Note that the midpoints of a semi-joined $P_4$ are weakly dominating vertices, and the endpoints are weakly isolated vertices.)

Most of this section will be devoted to establishing the following. 

\begin{thm}\label{thm: char via Construction Algorithm}
A graph $G$ is a weakly threshold graph if and only if $G$ can be obtained by beginning with a graph isomorphic to $K_1$ or to $P_4$ and  iteratively adding to the graph either a dominating vertex, an isolated vertex, a weakly dominating vertex, a weakly isolated vertex, or a semi-joined $P_4$.
\end{thm}

The additions in the theorem refer to new vertices added; in no case do we change any of the adjacency relationships that existed prior to the addition of a new vertex or set of vertices.

We present the proof of Theorem~\ref{thm: char via Construction Algorithm} in Section~\ref{subsec: pf of iter const}. In order to facilitate the proof and to lay groundwork for later sections, in Section~\ref{subsec: canon decomp} we introduce a decomposition scheme of graphs and degree sequences that will assist in analyzing the Erd\H{o}s--Gallai differences of a degree sequence. Following the proof of Theorem~\ref{thm: char via Construction Algorithm}, in Section~\ref{subsec: complements} we present a few of its consequences, including the fact that weakly threshold graphs are closed under complementation.

\subsection{The canonical decomposition of graphs and degree sequences} \label{subsec: canon decomp}

In this section we recall two composition operations and their accompanying decompositions. Both decompositions were introduced by Tyshkevich in multiple papers (our presentation is adapted from~\cite{Tyshkevich00}, which contains a summary and bibliography) and are called the \emph{canonical decomposition}; one is a decomposition of graphs, and the other is a natural translation of the ideas into the context of degree sequences. After describing the results of Tyshkevich, we use the canonical decomposition of a degree sequence to better understand its list of Erd\H{o}s--Gallai differences.

Let $(G,A,B)$ denote a triple consisting of a split graph $G$ and a  partition $A,B$ of its vertex set into an independent set $A$ and a clique $B$ such that either $A$ or $B$ can be empty but $A \cup B \neq \emptyset$. Since the partition $A,B$ is fixed, we refer to this triple as a \emph{splitted graph}. Similarly, if $d$ is the degree sequence of a split graph having partition $A,B$, then  we may form a \emph{splitted degree sequence} by writing the terms of $d$ with a semicolon separating the terms corresponding to vertices in $B$ from terms corresponding to vertices in $A$ (note that we can do this even while maintaining the terms in descending order). For example, the split graph with degree sequence $(3,3,2,1,1)$ has two distinct partitions of its vertex set into an independent set and a clique; the associated splitted degree sequences are $(3,3,2;1,1)$ and $(3,3;2,1,1)$.

Given a splitted graph $(G,A,B)$ and a graph $H$, each with nonempty vertex sets, we define the \emph{composition} $(G,A,B) \circ H$ to be the graph formed by taking the disjoint union of $G$ and $H$ and adding to it all edges joining vertices in $B$ to vertices in $V(H)$.

We can also compose two splitted graphs $(G,A,B)$ and $(H,C,D)$ by treating the second graph simply as a graph, though if we wish to we can also think of the resulting graph $(G,A,B) \circ (H,C,D)$ as a splitted graph with independent set $A \cup C$ and $B \cup D$. We trust that context will make it clear whether the result of a composition is intended as a (non-partitioned) graph or a splitted graph. With these understandings, however, we note that the operation $\circ$ is associative.

We use the same notation $\circ$ to indicate the analogous composition of a splitted degree sequence with a graph degree sequence or splitted degree sequence. Here, if $d=(d_1,\dots,d_k;d_{k+1},\dots,d_n)$ and $e = (e_1,\dots,e_m)$, then $d \circ e$ is obtained by adding $m$ to each of the terms $d_1,\dots,d_k$, adding $k$ to each of the terms $e_1,\dots,e_m$, and combining these results with the (unchanged) terms $d_{k+1},\dots,d_n$ into a sequence in descending order. In symbols, \[(d_1,\dots,d_k;d_{k+1},\dots,d_n) \circ (e_1,\dots,e_m) = (d_1+m,\dots,d_k+m, e_1+k,\dots,e_m+k,d_{k+1},\dots,d_n).\] Clearly, the composition of the splitted degree sequence of $(G,A,B)$ and the degree sequence of a graph $H$ is the degree sequence of the graph $(G,A,B) \circ H$. (We can analogously define the composition of two splitted degree sequences and treat these compositions with the same understandings as with the compositions of graphs.)

We call a (splitted or non-partitioned) graph or degree sequence \emph{decomposable} if it can be written as the composition of other graphs or sequences; otherwise, it is \emph{indecomposable}.

Tyshkevich showed that graphs and degree sequences can be completely decomposed in a unique way, which we refer to as the \emph{canonical decomposition}.

\begin{thm}[\cite{Tyshkevich00}]\label{thm: canon decomp}\mbox{}
\begin{enumerate}
\item[\textup{(i)}] Every graph $G$ can be expressed as a composition \[(G_k,A_k,B_k)\circ \dots \circ (G_1,A_1,B_1) \circ G_0\] of indecomposable components; here the leftmost $k$ components are splitted graphs and the rightmost is a graph (we say that $k=0$ if $G$ is indecomposable). Moreover, this decomposition is unique in the sense that two canonical decompositions of a graph have the same number of components and have isomorphisms between corresponding components (that preserve splitting partitions, in the case of splitted graph components).
\item[\textup{(ii)}] Every degree sequence $d$ can be uniquely expressed as a composition \[d=\alpha_k\circ \dots \circ \alpha_1 \circ \alpha_0\] of indecomposable components; here the leftmost $k$ components are splitted degree sequences and the rightmost is a degree sequence (we again say that $k=0$ if $d$ is indecomposable).
\end{enumerate}
\end{thm}

Our first application of the canonical decomposition will be to describe the Erd\H{o}s--Gallai differences of degree sequences. In~\cite{HeredUniII} the author presented a connection between the canonical components of a graph and the Erd\H{o}s--Gallai differences of its degree sequence that equal 0. We now describe a more general connection between Erd\H{o}s--Gallai differences and the composition operation $\circ$.

We begin by examining the corrected Ferrers diagrams of compositions. Suppose that $d$ is a splitted degree sequence with $n$ terms (with $k$ terms before its semicolon), and that $e$ is a degree sequence with $m$ terms. We form a new corrected Ferrers diagram by first dividing $C(d)$ into four blocks, with rows separated after the first $k$ rows and columns similarly separated. We move these blocks to the corresponding corners of a larger $(n+m)$-by-$(n+m)$ diagram, insert $C(e)$ in the middle of this diagram, and fill in the rest of the diagram with blocks of 1's above and to the left of the inserted copy of $C(e)$ and with blocks of 0's below and to the right. For example, if we let $d=(2,2;1,1)$ and $e=(1,1,1,1,0)$, then $C(d \circ e)$ is shown in Figure~\ref{fig: composition Ferrers}, with dashed lines illustrating the blocks (compare this diagram to those in Figure~\ref{fig: C(11110) and C(2211)}). It is straightforward to verify that if $b = d \circ e$, then the diagram constructed above is the corrected Ferrers diagram $C(b)$. 
\begin{figure}
\centering
$\displaystyle \left[\begin{array}{cc;{2pt/2pt}ccccc;{2pt/2pt}cc}
\star & 1 & 1 & 1 & 1 & 1 & 1 & 1 & 0\\
1 & \star & 1 & 1 & 1 & 1 & 1 & 1 & 0\\ \hdashline[2pt/2pt]
1 & 1 & \star & 1 & 0 & 0 & 0 & 0 & 0\\
1 & 1 & 1 & \star & 0 & 0 & 0 & 0 & 0\\
1 & 1 & 1 & 0 & \star & 0 & 0 & 0 & 0\\
1 & 1 & 1 & 0 & 0 & \star & 0 & 0 & 0\\
1 & 1 & 0 & 0 & 0 & 0 & \star & 0 & 0\\ \hdashline[2pt/2pt]
1 & 0 & 0 & 0 & 0 & 0 & 0 & \star & 0\\
1 & 0 & 0 & 0 & 0 & 0 & 0 & 0 & \star
\end{array}\right]$
\caption{The corrected Ferrers diagram of $(2,2;1,1)\circ(1,1,1,1,0) = (7,7,3,3,3,3,2,1,1)$.} \label{fig: composition Ferrers}
\end{figure}

Note now that $m(b) = k + m(e)$. By the symmetry of the blocks of 1's and of 0's which pad the copy of $C(e)$, Lemma~\ref{lem: EG diff via diagram} implies that $\Delta_i(b)=\Delta_i(d)$ for $1 \leq i \leq k$, and for $1 \leq i \leq m(e)$ we have $\Delta_{k+i}(b) = \Delta_i(e)$. Hence, the list of the first $m(b)$ Erd\H{o}s--Gallai differences of $b$ are obtained by appending the Erd\H{o}s--Gallai differences of $e$ to an initial segment of the list of Erd\H{o}s--Gallai differences of $d$.

More generally, we obtain the following by induction.

\begin{thm} \label{thm: concatenate EG diffs}
Suppose that $d$ is the degree sequence of a graph, and $d = \alpha_k \circ \dots \circ \alpha_1 \circ \alpha_0$ is the canonical decomposition of $d$, where in each $\alpha_i$ there are $m_i$ terms appearing before the semicolon. 

Suppose also that $S(d)$ is the sequence of integers beginning with the first $m_k$ Erd\H{o}s--Gallai differences of $\alpha_k$, followed by the first $m_{k-1}$ Erd\H{o}s--Gallai differences of $\alpha_{k-1}$, and so on, up through the first $m_1$ Erd\H{o}s--Gallai differences of $\alpha_1$, and ending with the first $m(\alpha_0)$ Erd\H{o}s--Gallai differences of $\alpha_0$.

The terms of $S(d)$, in this order, are precisely the first $m(d)$ Erd\H{o}s--Gallai differences of $d$.
\end{thm}

\subsection{Proof of Theorem~\ref{thm: char via Construction Algorithm}} \label{subsec: pf of iter const}

We first prove that if a graph is weakly threshold then it can be built up through the operations described in the theorem. We begin with a lemma.

\begin{lem} \label{lem: max min degrees}
If $d=(d_1,\dots,d_n)$ is the degree sequence of a  weakly threshold graph $G$, then exactly one of the following is true:
\begin{enumerate}
\item[\textup{(a)}] $d_n=0$ and $G$ has an isolated vertex;
\item[\textup{(b)}] $d_1=n-1$ and $G$ has a dominating vertex;
\item[\textup{(c)}] $d_2 < d_1 = n-2$ and $d_{n-1} = d_n = 1$, and $G$ has a weakly isolated vertex;
\item[\textup{(d)}] $d_1 = d_2 = n-2$ and $d_{n-1} > d_n = 1$, and $G$ has a weakly dominating vertex;
\item[\textup{(e)}] $d_1=d_2=n-2$ and $d_{n-1}=d_n=1$ and $\Delta_2(d)=0$, and $G$ has four vertices forming a $P_4$ that is semi-joined to the rest of the graph.
\end{enumerate}
In each case, deleting the vertex (or four vertices, in the last case) described leaves a graph that is also weakly threshold.
\end{lem}
\begin{proof}
Comparing the degree conditions in each case, we see that no two of the properties (a)--(e) can simultaneously hold.

Suppose that neither (a) nor (b) holds. Then $d_1 \leq n-2$ and $d_n \geq 1$; hence 
\[1 \geq \Delta_1(d) = 1\cdot 0 + (n-1)\cdot 1 - d_1,\] and $d_1=n-2$. We know $d_2 \neq 1$, since otherwise the terms of $d$ would then sum to an odd number, contradicting our assumption that $d$ is a degree sequence. Letting $p$ denote the number of terms of $d$ exactly equal to $1$, we can now write \[1 \geq \Delta_2(d) = 2 \cdot 1 + (n-2-p)\cdot 2 + p \cdot 1 - (n-2)-d_2.\] Then $n-2 = d_1 \geq d_2 \geq n-1-p$, so $p \geq 1$ and hence $d_n=1$. In fact, if $d_2 < d_1$ then $d_{n-1}=d_n=1$, so (c) holds; if instead $d_{n-1}>d_n$, then $d_2=d_1$ and (d) holds. Finally, if $d_1=d_2=n-2$ and $d_{n-1}=d_n = 1$, then $p \geq 2$. Since $d$ is a degree sequence, we also have $0 \leq \Delta_2(d) = 2-p$; we see that in fact $\Delta_2(d)=0$, and property (e) holds.
\end{proof}

Lemma~\ref{lem: max min degrees} allow us to apply induction on the number of vertices in a weakly threshold graph. Each weakly threshold graph on up to four vertices is either isomorphic to $P_4$ or is a threshold graph, in which case Theorem~\ref{thm: threshold dom/iso} implies that the graph can be built up from $K_1$ by adding dominating and/or isolated vertices.

Suppose now that $G$ is an arbitrary weakly threshold graph on $n \geq 5$ vertices, and that every weakly threshold graph on fewer than $n$ vertices can be built up from $K_1$ or $P_4$ by iteratively adding a vertex or vertices as claimed. It follows from Lemma~\ref{lem: max min degrees} that $G$ may be obtained by adding a dominating, isolated, weakly dominating, or weakly isolated vertex, or adding a semi-joined $P_4$, to a weakly threshold graph on fewer than $n$ vertices, and so by the induction hypothesis $G$ can be built up in the desired way from $K_1$ or $P_4$. This completes our proof that weakly threshold graphs may each be constructed using the operations from Theorem~\ref{thm: char via Construction Algorithm}.

\bigskip
In order to prove the converse, we first present a simplifying lemma. Let $\mathcal{C}$ denote the set of graphs that may be built up from $K_1$ or $P_4$ through the operations described in Theorem~\ref{thm: char via Construction Algorithm}. Call these operations (adding a dominating vertex, an isolated vertex, a weakly dominating vertex, a weakly isolated vertex, or a semi-joined $P_4$) \emph{permissible operations}. Call the addition of a dominating vertex or an isolated vertex a Type 1 operation, call the addition of a weakly dominating vertex or weakly isolated vertex a Type 2 operation, and call the addition of a semi-joined $P_4$ a Type 3 operation. (We will use these terms in Sections~\ref{sec: forb subgr} and~\ref{sec: enum} as well.)

\begin{lem} \label{lem: weaklies placement}
For any element $G$ of $\mathcal{C}$, there exists a sequence of permissible operations that constructs $G$ from $K_1$ or $P_4$ with the property that between any Type 1 operation and a later Type 2 operation, a Type 3 operation is performed.\end{lem}
\begin{proof}
We proceed by induction on the number $p$ of permissible operations needed to construct $G$; fix a sequence of operations $O_1,\dots,O_p$ that constructs $G$ from $K_1$ or $P_4$. If $p<2$ then the conclusion holds trivially.

Now suppose that $p=k+1$ for some integer $k \geq 1$, and all graphs in $\mathcal{C}$ that can be constructed from $K_1$ or $P_4$ using $k$ permissible operations can be constructed so that between any Type 1 operation and a later Type 2 operation, there occurs a Type 3 operation.

Let $G'$ be the graph on which the operation $O_{k+1}$ is performed to create $G$; by the induction hypothesis, we may assume that in the construction of $G'$ at least one Type 3 operation occurs between any Type 1 operation and a later Type 2 operation. The conclusion of the lemma holds for $G$ except possibly in the case that $O_{k+1}$ is a Type 2 operation and $O_{k}$ is a Type 1 operation, so assume that $O_k$ and $O_{k+1}$ are operations of these types. Further let $G''$ be the graph on which the operation $O_{k}$ is performed to create $G'$.

If $O_k$ is the addition of an isolated vertex and $O_{k+1}$ is the addition of a weakly isolated vertex, then $G$ can be formed by first adding a weakly isolated vertex to $G''$ (call the resulting graph $G^*$) and then adding an isolated vertex. The induction hypothesis applies to $G^*$, so some sequence of operations creating $G^*$ has Type 3 operations in all the appropriate places; this sequence, followed by adding an isolated vertex, is a sequence of operations creating $G$ that satisfies the claim of the lemma. A similar argument applies if $O_k$ and $O_{k+1}$ are the additions of a dominating vertex and a weakly dominating vertex, respectively.

If $O_k$ is the addition of an isolated vertex and $O_{k+1}$ is the addition of a weakly dominating vertex, then $G$ can be created by adding a dominating vertex to $G''$ (again call the resulting graph $G^*$) and then adding an isolated vertex. As before, we obtain a suitable construction of $G$ by appending the addition of an isolated vertex to a suitable construction of $G^*$. A similar argument handles the case that $O_k$ is the addition of a dominating vertex and $O_{k+1}$ is the addition of a weakly isolated vertex, completing the inductive step.
\end{proof}

We can now prove that every graph in $\mathcal{C}$ is weakly threshold. We proceed by induction on the number of vertices. 

By Theorem~\ref{thm: threshold dom/iso}, any graph in $\mathcal{C}$ on four or fewer vertices is either a threshold graph or $P_4$ and hence must be weakly threshold.

Now suppose that every graph in $\mathcal{C}$ with fewer than $n$ vertices, where $n \geq 5$, is weakly threshold, and let $G$ be a graph in $\mathcal{C}$ with $n$ vertices. Consider a sequence of permissible operations that produces $G$ from $K_1$ or $P_4$ and that has the property described in Lemma~\ref{lem: weaklies placement}. We proceed by cases according to the last-performed operation. Let $d=(d_1,\dots,d_n)$ be the degree sequence of $G$. 

\smallskip
\noindent \emph{Case: The last operation in creating $G$ is a Type 1 operation or a Type 3 operation.} 

Observe that graphs with a dominating vertex, isolated vertex, or semi-joined $P_4$ are all decomposable under $\circ$. Let $d$ be the degree sequence of $G$, and suppose that $G'$ is the graph that the last operation is performed on to yield $G$, and that $d'$ is the degree sequence of $G'$. If the last operation in constructing $G$ is the addition of a dominating vertex, then $d = (0;)\circ d'$. If the last operation is the addition of an isolated vertex, then $d=(;0)\circ d'$. Finally, if the last operation is the addition of a semi-joined $P_4$, then $d=(2,2;1,1)\circ d'$. By the induction hypothesis, $G'$ is a weakly threshold graph, so the first $m(d')$ Erd\H{o}s differences of $d'$ are all $0$ or $1$. Let $s'$ denote the list of these differences, and suppose that $s$ is the list of the first $m(d)$ Erd\H{o}s--Gallai differences of $G$. It follows from Theorem~\ref{thm: concatenate EG diffs} that adding a dominating vertex forms $s$ by inserting a 0 at the beginning of $s'$, adding an isolated vertex yields $s=s'$, and adding a semi-joined $P_4$ forms $s$ by inserting the terms $1,0$ at the beginning of $s'$. In each case, each entry of $s$ is 0 or 1, so $G$ is weakly threshold.

\smallskip
\noindent \emph{Case: The last operation in creating $G$ adds a weakly dominating vertex.} 

Let $v$ denote the added vertex, and let $G'=G-v$. The degree sequence of $G'$ is $d'=(d_2-1,\dots,d_{n-1}-1,d_n)$. By hypothesis, $G'$ is a weakly threshold graph, so the first $m(d')$ Erd\H{o}s differences of $d'$ are all $0$ or $1$.

By Lemma~\ref{lem: weaklies placement}, we may assume that a Type 3 operation was performed after the last Type 1 operation, if any Type 1 operation occurred; if none did, we may assume that the construction of $G$ began with $P_4$. Note now that when the construction algorithm began with $P_4$, or the last Type 3 operation was employed, the immediately resulting graph had minimum degree 1 and maximum degree equal to 2 less than the number of vertices. These properties are preserved by any Type 2 operations that follow, so we may assume that $G'$ has minimum degree 1 and maximum degree $n-3$.

We now compare the corrected Ferrers diagrams of $d$ and of $d'$. Observe that we may obtain $C(d)$ by first taking $C(d')$ and inserting a new first row and column containing $n-1$ copies of 1, as shown in the first diagram in Figure~\ref{fig: weakly dom diagram}. (In the diagrams we have shown a few specific entries, to emphasize the maximum and minimum degree in $d'$.) \begin{figure}
\centering
$\displaystyle  \left[\begin{array}{c;{2pt/2pt}ccccc}
\star & 1 & 1 & \cdots & 1 & 0\\ \hdashline[2pt/2pt]
1 & & & & 1 & 0\\
1 & & & & & \\
\vdots &  &  & C(d') & & \\
1 & & & & & \\
0 & 1 & 0 & & & \\
\end{array}\right] \qquad \qquad \left[\begin{array}{ccccc;{2pt/2pt}c}
 & & & 1 & 0 & 1\\
 & & & & & 0\\
 & & C(d') & & & \vdots\\
 & & & & & 0\\
1 & 0 & & & & 0\\
\hdashline[2pt/2pt]
1 & 0 & \cdots & 0 & 0 & \star
\end{array}\right]$
\caption{Additions to $C(d')$ in the construction of $C(d)$.} \label{fig: weakly dom diagram}
\end{figure}
The last row of $C(d')$ contains a single 1, followed by 0's and a terminal $\star$, and the last column of $C(d')$ contains only 0's and the final $\star$. Thus to complete the creation of $C(d)$ from $C(d')$, we interchange the 0 and the 1 in the last row of the augmented diagram.

Observe that $m(d) = m(d')+1$. Applying Lemma~\ref{lem: EG diff via diagram}, we see that $\Delta_1(d) = 1$, and $\Delta_{i}(d) = \Delta_{i-1}(d')$ for each $i$ such that $2 \leq i \leq m(d)$. It follows that $G$ is a weakly threshold graph.  

\smallskip
\noindent \emph{Case: The last operation in creating $G$ adds a weakly isolated vertex.} 

As before, let $v$ denote the added vertex, and let $G'=G-v$. By the induction hypothesis, $G'$ is a weakly threshold graph, and the first $m(d')$ Erd\H{o}s differences of $d'$ are all $0$ or $1$. By the same argument as in the previous case, $G'$ has minimum degree 1 and maximum degree $n-3$.

We may obtain $C(d)$ in this case by first taking $C(d')$ and appending a new row and column each containing one copy of 1, as shown in the second diagram in Figure~\ref{fig: weakly dom diagram}. As a reminder, the diagram shows the single 0 in the top row of $C(d')$. To create $C(d)$, we interchange the 0 and the 1 in the first row of the augmented diagram.

Observe that $m(d) = m(d')$, and that by Lemma~\ref{lem: EG diff via diagram}, $\Delta_{i}(d) = \Delta_{i}(d')$ for each $i$ such that $1 \leq i \leq m(d)$. It follows that $G$ is a weakly threshold graph, and our proof of Theorem~\ref{thm: char via Construction Algorithm} is complete.

\subsection{Weakly threshold graphs and complementation} \label{subsec: complements}

The iterative construction in Theorem~\ref{thm: char via Construction Algorithm} allows us an easy conclusion about weakly threshold graphs and sequences not necessarily obvious from their definitions. Henceforth, let $\overline{G}$ denote the complement of a graph $G$. Also let $G\vee H$ and $G+H$ denote the join and disjoint union, respectively, of graphs $G$ and $H$. It is easy to see that for any graphs $G$ and $H$, $\overline{G \vee H} \cong \overline{G} + \overline{H}$.

\begin{thm}\label{thm: complementation}
A graph is a weakly threshold graph if and only if its complement is.
\end{thm}
\begin{proof}
The result follows by induction on the number of addition operations needed to construct a weakly threshold graph; first note that $K_1$ and $P_4$ are self-complementary. Adding a dominating vertex to a graph $G$ has the effect of simultaneously adding an isolated vertex to $\overline{G}$, i.e., $\overline{G \vee K_1} \cong \overline{G} + K_1$. Similarly, $\overline{G + K_1} \cong \overline{G} \vee K_1$, and additions of weakly dominating vertices and weakly isolated vertices have the same relationship. Finally, because $\overline{P_4} \cong P_4$, and complementation changes endpoints to midpoints and vice versa, adding a semi-joined $P_4$ to a graph has the effect of adding a semi-joined $P_4$ to the complement. Thus a graph can iteratively be constructed using these types of operations if and only if its complement can.
\end{proof}

\section{A forbidden subgraph characterization} \label{sec: forb subgr}

In this section we show that the weakly threshold graphs form a hereditary graph class, i.e., the property of being a weakly threshold graph is preserved under taking induced subgraphs. This allows us to characterize these graphs in terms of a collection of minimal forbidden induced subgraphs, and it reveals a connection between weakly threshold graphs and interval graphs.

Given a graph $F$, we say that a graph $G$ is \emph{$F$-free} if no induced subgraph of $G$ is isomorphic to $F$. If $\mathcal{F}$ is a collection of graphs, then $G$ is \emph{$\mathcal{F}$-free} if $G$ is $F$-free for every element $F$ of $\mathcal{F}$. Let $2K_2$ denote $K_2+K_2$, and let $H$ and $S_3$ respectively denote the unique split graphs with degree sequences $(3,3,1,1,1,1)$ and $(3,3,3,1,1,1)$.

\begin{thm} \label{thm: forb subgr}
A graph $G$ is a weakly threshold graph if and only if $G$ is $\{2K_2,C_4,C_5,H,\overline{H}, S_3,\overline{S_3}\}$-free.
\end{thm}
\begin{proof}
In the following, let $\mathcal{F} = \{2K_2,C_4,C_5,H,\overline{H}, S_3,\overline{S_3}\}$. 

Suppose first that $G$ is a weakly threshold graph. By Theorem~\ref{thm: WT graphs are split}, $G$ is a split graph. Since all split graphs are $\{2K_2,C_4,C_5\}$-free (this was proved by F\"{o}ldes and Hammer in~\cite{FoldesHammer76}), $G$ induces none of these three subgraphs. 

By Theorem~\ref{thm: char via Construction Algorithm} we know that there is a sequence of operations $\mathcal{O}_1,\dots,\mathcal{O}_p$ of Types 1, 2, or 3 (as defined in the previous section) that create $G$ from $K_1$ or $P_4$. We prove that $G$ is $\mathcal{F}$-free by induction on $p$. Observe that if $p=0$, then $G$ is $K_1$ or $P_4$, both of which are $\mathcal{F}$-free.

Suppose now that $p=k+1$ for some nonnegative integer $k$, and assume that every weakly threshold graph that can be constructed from $K_1$ or $P_4$ via a sequence of $k$ addition operations is $\mathcal{F}$-free. Let $G'$ be the graph from which $G$ is created by applying the operation $\mathcal{O}_p$. By assumption, $G'$ is $\mathcal{F}$-free.

If $\mathcal{O}_p$ is a Type 1 or Type 3 operation, then $G = (G_1,A_1,B_1) \circ G'$, where $G_1$ is isomorphic to $K_1$ or $P_4$. Since $G_1$ and $G'$ are both $\mathcal{F}$-free, and we can verify that every graph in $\mathcal{F}$ is indecomposable, it follows that $G$ is $\mathcal{F}$-free as well.

Suppose instead that $\mathcal{O}_p$ is a Type 2 operation, and that $v$ is the vertex that is added to $G'$ to create $G$. Suppose also to the contrary that $G$ does induce an element of $\mathcal{F}$ other than $2K_2$, $C_4$, or $C_5$. Since this induced subgraph was not present in $G'$, it must contain the vertex $v$. Let $A,B$ be a partition of $V(G)$ into an independent set and a clique, respectively.

If $G$ contains an induced subgraph $F$ isomorphic to $H$, then the vertices of degree 1 in $F$ must belong to $A$, and the two other vertices belong to $B$. In the operation $\mathcal{O}_p$ the vertex $v$ cannot have been a weakly dominating vertex, since $v$ would have to be a dominating or weakly dominating vertex in $F$, and $H$ has no such vertex. Thus $v$ is a weakly isolated vertex in $G$ and hence is one of the vertices of $F$ in $A$. Let $w$ denote the neighbor of $v$ in $F$, and let $x$ denote the other vertex of degree 3 in $F$. Since $v$ was added to $G'$ as a weakly isolated vertex, this implies that $w$ was a vertex of maximum degree in $G'$, so in $G$ the degree of $w$ is larger than the degree of $x$. Since in $F$ the vertex $x$ has the same degree as $w$, the vertex $w$ must have a neighbor $y$ that $x$ does not; this vertex must belong to $A$, along with the vertices of degree 1 in $F$. However, the vertex $y$, together with the vertices of $F-v$, then induces $H$ in $G'$, a contradiction to the induction hypothesis.

If instead $G$ contains an induced subgraph $F$ isomorphic to $S_3$, then again the vertices of degree 1 in $F$ belong to $A$, while the vertices of degree 3 belong to $B$. Since $F$ contains no dominating or weakly dominating vertex, then as before, vertex $v$ was added during $\mathcal{O}_p$ as a weakly isolated vertex, so $v$ is one of the vertices of $F$ in $A$. Let $w$ be the neighbor of $v$ in $F$, and denote the other vertices of $F$ in $B$ by $x$ and $x'$. Since $v$ was added as a weakly isolated vertex in $\mathcal{O}_p$, vertex $w$ has a higher degree in $G$ than do $x$ or $x'$. Since in $F$ the vertices $x$ and $x'$ have the same degree as $w$, each of these vertices must be non-adjacent to some neighbor of $w$ other than $v$. If some neighbor $y$ of $w$ other than $v$ is non-adjacent to both $x$ and $x'$, then $y$, together with the vertices of $F-v$, induces $S_3$ in $G'$, a contradiction to the induction hypothesis. Otherwise, $w$ has neighbors $y$, which is adjacent to $x$ but not $x'$, and $y'$, which is adjacent to $x'$ but not $x$. However, the vertices $y,y',w,'w$, and the two vertices non-adjacent to $w$ in $F$ then induce $H$ in $G-v$, which is another contradiction.

Finally, if $G$ contains an induced subgraph $F$ isomorphic to $\overline{H}$ or $\overline{S_3}$, then by Theorem~\ref{thm: complementation} we can apply the arguments of the last two paragraphs to $\overline{G}$, which must contain $H$ or $S_3$, to arrive at a similar contradiction. From these contradictions in every case we conclude that $G$ is $\mathcal{F}$-free, and in fact all weakly threshold graphs are as well.

We now prove that all $\mathcal{F}$-free graphs are weakly threshold graphs. We do this by induction on the number $n$ of vertices in an arbitrary $\mathcal{F}$-free graph $G$. Observe that all $\mathcal{F}$-free graphs on at most four vertices are threshold graphs or are isomorphic to $P_4$; any such graph is a weakly threshold graph.

Suppose now that $n \geq 5$, and assume that every $\mathcal{F}$-free graph with fewer than $n$ vertices is a weakly threshold graph. Note that if $G$ contains a dominating or isolated vertex $v$, then by the induction hypothesis the graph $G-v$ is a weakly threshold graph. As such it can be constructed by a sequence of operations as described in Theorem~\ref{thm: char via Construction Algorithm}; if we append to this sequence the addition of $v$ to the graph (a Type 1 operation), then Theorem~\ref{thm: char via Construction Algorithm} implies that $G$ is a weakly threshold graph as well. A similar conclusion holds if $G$ contains a semi-joined $P_4$. Suppose now that $G$ has no dominating or isolated vertex, and no semi-joined $P_4$. 

Observe that since $G$ is $\{2K_2,C_4,C_5\}$ free, $G$ is split (see~\cite{FoldesHammer76}). Fix a partition $A,B$ of $V(G)$ into an independent set and a clique, respectively.

Let $v$ be a vertex of maximum degree in $G$; since $G$ is split, we may assume that $v$ is an element of $B$. We claim that the degree of $v$ is $n-2$. If this is not the case, then $v$ is non-adjacent to at least two vertices $w_1$ and $w_2$, which must belong to $A$. Since $G$ has no isolated vertices, the vertices $w_1$ and $w_2$ each have a neighbor in $B$. If they have a common neighbor $x$, then since $v$ had maximum degree, $v$ must have two neighbors that $x$ is not adjacent to. These two neighbors then must belong to $A$, and together with $v,w_1,w_2,x$ induce a subgraph isomorphic to $H$, a contradiction. If instead $w_1$ and $w_2$ have no common neighbor, then $w_1$ is adjacent to $x_1$, and $w_2$ is adjacent to $x_2$ for some $x_1,x_2 \in B$. Since $v$ has the maximum degree in $G$, this implies that $x_1$ and $x_2$ each have a non-neighbor among the neighbors of $v$. If $x_1$ and $x_2$ have such a non-neighbor in common, then this vertex and vertices  $v,w_1,w_2,x_1,x_2$ induce a copy of $S_3$ in $G$, a contradiction. Thus $v$ has a neighbor $y_1$ adjacent to $x_1$ but not $x_2$, and a neighbor $y_2$ adjacent to $x_2$ but not $x_1$. However, then the vertices $w_1,w_2,x_1,x_2,y_1,y_2$ together induce a copy of $H$, again a contradiction. In light of all these contradictions, we conclude that $v$ cannot have two non-neighbors in $G$; hence, $v$ has degree $n-2$.

Let $z$ be a vertex of minimum degree in $G$. Note that $z$ is a vertex of maximum degree in the complement $\overline{G}$, and $\overline{G}$ is also split and $\mathcal{F}$-free since $\mathcal{F}$ is closed under taking complements. Furthermore, $\overline{G}$ also cannot contain any dominating or isolated vertices. Thus the arguments above show that $z$ has degree $n-2$ in $\overline{G}$ and hence $z$ has degree $1$ in $G$.

If $G$ has two vertices $v,v'$ of degree $n-2$ and two vertices $z,z'$ of degree 1, then we can verify that either $G$ is isomorphic to $P_4$ or the subgraph induced by $v,v',z,z'$ is a semi-joined $P_4$ in $G$. Since both these possibilities have already been handled previously, we assume now that $G$ has either a unique vertex of degree $n-2$ or a unique vertex of degree $1$.

Suppose that $G$ has a unique vertex $v$ of maximum degree $n-2$. If $v$ is adjacent to any vertex $z$ of degree 1, then we may obtain $G$ from the graph $G-z$ by adding a weakly dominating vertex (namely, attaching vertex $z$ to $v$). If $v$ is not adjacent to any vertex of degree 1, then the minimum degree of $G-v$ is still 1, and we may obtain $G$ again by adding a weakly dominating vertex.

We may apply similar arguments to the $\mathcal{F}$-free graph $\overline{G}$; we conclude that since $G$ has either a unique vertex of degree $n-2$ or a unique vertex of degree 1, $G$ may be obtained from $G-w$, where $w$ is some vertex of $G$, via a Type 2 operation. Since $G-w$ is $\mathcal{F}$-free, by the induction hypothesis $G-w$ is a weakly threshold graph and thus Theorem~\ref{thm: char via Construction Algorithm} applies; if we append the Type 2 operation that replaces $w$ to the sequence of permissible operations that constructs $G-w$, we see that $G$ can also be constructed from $K_1$ or $P_4$ by a sequence of permissible operations, implying that $G$ is weakly threshold.
\end{proof}

Interestingly, the list $\mathcal{F}$ of forbidden subgraphs is strikingly similar to that of another hereditary family. Let $\mathcal{H}$ denote the class of graphs that are both interval graphs and complements of interval graphs. As noted in~\cite{ISGCI}, this class is equivalent to the class of split permutation graphs and is precisely the class of $\{2K_2,C_4,C_5,S_3,\overline{S_3},\text{rising sun},\text{co-rising sun}\}$-free, where the rising sun and co-rising sun graphs are shown in Figure~\ref{fig: rising sun}. 
\begin{figure}
\centering
\includegraphics[width=0.5\textwidth]{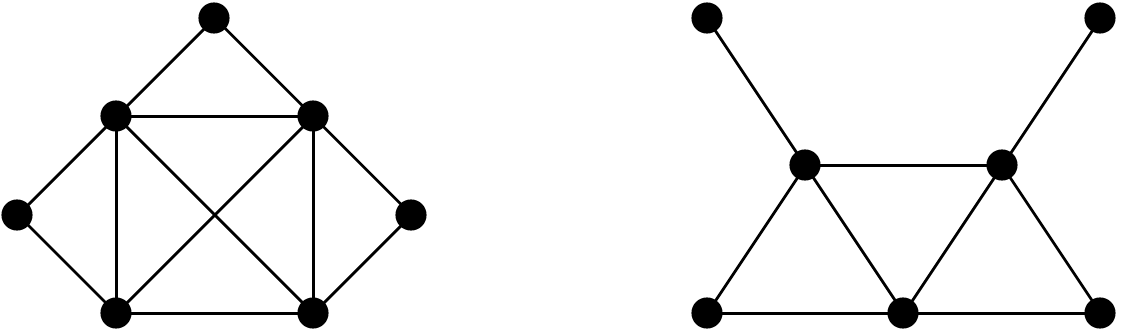}
\caption{The rising sun graph and its complement.} \label{fig: rising sun}
\end{figure}
Note that $H$ and $\overline{H}$ can be obtained by deleting a single vertex from the rising sun graph and from its complement, respectively. Hence the weakly threshold graphs form a notable subclass of $\mathcal{H}$.

\section{Enumeration} \label{sec: enum}


In this section we count both the weakly threshold sequences and the weakly threshold graphs of order $n$. Our approach, which is essentially the same technique used by Tyshkevich in~\cite{Tyshkevich84} in enumerating matrogenic and matroidal graphs, will use the canonical decomposition as a convenient framework. For both sequences and graphs, we begin by finding recurrences that are satisfied respectively by the number of $n$-term sequences and $n$-vertex graphs in question. Using the generating functions of these sequences,  together with the structure imposed in both contexts by the canonical decomposition, we obtain generating functions for the number of weakly threshold sequences and graphs, which we denote by $S(x)$ and by $W(x)$, respectively.

For $n \geq 4$, let $g_n$ be the number of indecomposable weakly threshold sequences with $n$ terms, and let $h_n$ be the number of indecomposable weakly threshold graphs with $n$ vertices. If \[G(x) = 2x + \sum_{k=4}^\infty g_kx^k  \qquad \text{ and } \qquad H(x) = 2x  + \sum_{k=4}^\infty h_k x^k,\] then $G(x)$ and $H(x)$ are the generating functions for the number of splitted indecomposable weakly threshold sequences and graphs, respectively, indexed by the number of terms or vertices. Note that in both equations the coefficient of $2$ in front of $x$ reflects that in a canonical component having a single vertex, this vertex may belong to either the clique or the independent set in the splitted graph.

We now discuss weakly threshold graphs, though analogous arguments apply to weakly threshold sequences. The sequence counting the $n$-vertex weakly threshold graphs having exactly $k$ canonically indecomposable components has generating function $W_k(x)$ given by \[W_k(x) = H(x)^{k-1}(H(x)-x),\] since as the distributive law is applied to the product of sums, the coefficient on the resulting term with degree $n$ counts the ways to choose the $k$ indecomposable components for the canonical decomposition. Note here that the last factor in the expression above is $H(x)-x$, since if the rightmost component in the canonical decomposition is isomorphic to $K_1$, the isomorphism class is the same whether the vertex is in the clique or the independent set of the splitted graph.

Summing the functions $W_k(x)$, we arrive at the generating function $W(x)$ for the number of weakly threshold graphs:
\begin{equation}\label{eq: W(x)}
W(x) = \sum_{k=1}^\infty W_k(x) = \sum_{k=1}^\infty H(x)^{k-1}(H(x)-x) = \frac{H(x)-x}{1-H(x)}.\end{equation}
Similarly, \begin{equation}\label{eq: S(x)}
S(x) = \frac{G(x)-x}{1-G(x)}.\end{equation}

(Here and elsewhere we assume that $x$ belongs to some suitable interval of convergence.) In Sections~\ref{subsec: sequences} and~\ref{subsec: graphs} we will derive expresions for $G(x)$ and $H(x)$, respectively. Then, in Section~\ref{subsec: conclusion}, we will use~\eqref{eq: W(x)} and~\eqref{eq: S(x)} to obtain the  generating functions $S(x)$ and $G(x)$ and comment on the numbers of weakly threshold sequences and graphs.

\subsection{Indecomposable weakly threshold sequences} \label{subsec: sequences}

In order to determine $G(x)$ we derive a recurrence for the sequence $(g_n)$. By direct observation we see that $g_4=1$, since $(2,2,1,1)$ is the unique indecomposable weakly threshold sequence with four terms. It follows from Theorem~\ref{thm: char via Construction Algorithm} that for any $n \geq 5$ we may obtain any $n$-term indecomposable weakly threshold sequence by choosing an $(n-1)$-term indecomposable weakly threshold sequence and either (i) appending a 1 to the end of the sequence and increasing the first term of the sequence by 1, or (ii) increasing the first $n-2$ terms of the sequence by 1 and then inserting another copy of the value $n-2$ at the beginning of the sequence. Thus $g_n = 2g_{n-1}$ and hence $g_n = 2^{n-4}$ for $n \geq 4$. We see that

\begin{equation}
\label{eq: G(x)}
G(x) = 2x + \sum_{k=4}^\infty 2^{k-4}x^k =2x + \frac{x^4}{1-2x} = \frac{2x - 4x^2+x^4}{1-2x}.
\end{equation}

\subsection{Indecomposable weakly threshold graphs} \label{subsec: graphs}

We now count the isomorphism classes of indecomposable weakly threshold graphs on $n$ vertices. The arguments here will be more intricate than in Section~\ref{subsec: sequences}, since more than one weakly threshold graph can have the same degree sequence. In order to obtain a recurrence for $h_n$, we will first need stronger results on the iterative construction of indecomposable weakly threshold graphs, which we present in Lemmas~\ref{lem: constructing indecomp WT} and~\ref{lem: module cases}. As in the last section, we then use the recurrence to derive a closed form expression for the generating function $H(x)$.

As a preliminary step, we recall a characterization of indecomposable graphs due to the author and West.

\begin{lem}[{{\cite[Theorem 3.2]{BarrusWest12}}}] \label{lem: A4struc}
A graph $G$ is canonically indecomposable if and only if for every pair $u, v$ of vertices there is a sequence $A_1,\dots,A_k$ of 4-element vertex subsets of $G$ such that $u$ and $v$ belong to $A_1$ and $A_k$, respectively, consecutive subsets in the sequence have nonempty intersection, and each $A_i$ is the vertex set of an induced subgraph isomorphic to $2K_2$, $C_4$, or $P_4$.
\end{lem}

\begin{lem}\label{lem: constructing indecomp WT}
A graph $G$ is a canonically indecomposable weakly threshold graph if and only if $G$ is isomorphic to $K_1$ or can be obtained by beginning with a graph isomorphic to $P_4$ and iteratively performing a sequence of Type 2 operations.
\end{lem}
\begin{proof}
Suppose that $G$ is an indecomposable weakly threshold graph. By Theorem~\ref{thm: char via Construction Algorithm} and Lemma~\ref{lem: weaklies placement} we may assume that there exists a sequence $O_1,\dots,O_p$ of operations, each of which is Type 1, Type 2, or Type 3, that constructs $G$ from $K_1$ or $P_4$, where a Type 3 operation occurs between any Type 1 operation that is followed later in the sequence by a Type 2 operation. It is straightforward to see by induction on $p$ that if $G$ is constructed from $K_1$ and $p>1$, or if $G$ is constructed from $P_4$ and any of the $O_i$ is a Type 1 or Type 3 operation, then $G$ is canonically decomposable, a contradiction. Thus $O_1,\dots,O_p$ is a sequence of Type 2 operations that construct $G$ from $P_4$.

Conversely, noting that $K_1$ is an indecomposable weakly threshold graph, suppose that $G$ is constructed from $P_4$ via a sequence $O_1,\dots,O_p$ of Type 2 operations. By Theorem~\ref{thm: char via Construction Algorithm} we see that $G$ is a weakly threshold graph. We now prove by induction on $p$ that $G$ is indecomposable. This is true if $p=0$, since $G$ is then isomorphic to $P_4$. Suppose that all graphs constructed from $P_4$ via $k$ Type 2 operations are indecomposable, where $k$ is some nonnegative integer, and suppose that $p=k+1$. Now consider the vertex $u$ added during the operation $O_{p}$. If $u$ is a weakly dominating vertex, then there is a vertex $v$ of degree 1 not adjacent to $u$. Let $w$ be the neighbor of $v$. Since $u$ is weakly dominating, $u$ is adjacent to $w$, and since $u$ has a degree at least as large as $w$, while $w$ has a neighbor that $u$ does not, $u$ must have a neighbor $x$ that $w$ does not. Recalling Theorem~\ref{thm: forb subgr}, we conclude that the subgraph of $G$ induced on $\{u,v,w,x\}$ is isomorphic to $P_4$. By assumption the graph $G-u$ is canonically indecomposable; it follows from Lemma~\ref{lem: A4struc} that $G$ is indecomposable as well. A similar argument holds if $u$ is a weakly isolated vertex, so by induction we conclude that $G$ must be an indecomposable weakly threshold graph.
\end{proof}

In preparation for Lemma~\ref{lem: module cases}, we now present structural results on modules in canonically indecomposable split graphs. A \emph{module} in a graph $G$ is a set $M$ of vertices of $G$ such that for each $v \in V(G)-M$, the vertex $v$ is adjacent to either all or none of the vertices of $M$. The module $M$ is \emph{proper} if $M \neq V(G)$.

A well known theorem of Gallai~\cite{Gallai67} states that if a graph is neither disconnected nor the complement of a disconnected graph, then its maximal proper modules are disjoint; this result is the foundation for what is known as the \emph{modular decomposition} of a graph. The modular and canonical decompositions of a graph are usually distinct, though one connection is easy to verify: if $G$ has canonical decomposition $(G_k,A_k,B_k)\circ \dots \circ (G_1,A_1,B_1) \circ G_0$, then each set of the form $\bigcup_{i=0}^j V(G_i)$ is a module.

Despite this connection and our interest in modules in the results of this section, in what follows we will not refer to the modular decomposition of $G$, other than a quick application of Gallai's result during the proof of Lemma~\ref{lem: modules and twin sets}. As in previous sections, the terms `decomposable' and `indecomposable' will refer solely to the canonical decomposition.

\begin{lem}\label{lem: max modules}
If $G$ is a canonically indecomposable split graph, then every maximal proper module $M$ of $G$ lies within the maximum clique or within the maximum independent set of $G$. \end{lem}
\begin{proof}
Let $G$ be a canonically indecomposable split graph with maximum clique $Q$ and maximum independent set $I$, and suppose to the contrary that there exists a maximal proper module $M$ such that $M$ includes a vertex $x$ from $Q$ and a vertex $y$ from $I$. Note that every vertex from $Q - M$ is adjacent to $x$ and hence must be adjacent to $y$, and every vertex from $I-M$ is non-adjacent to $y$ and hence must be non-adjacent to $x$. However, then we may write $(G,I,Q) = (G-M, I-M,Q-M) \circ (G[M], I \cap M, Q \cap M)$, which is a contradiction, since $M$ is a nonempty proper module and $G$ is indecomposable.
\end{proof}

Two or more pairwise nonadjacent vertices in a graph are \emph{twins} if they have the same neighbors. Two or more pairwise adjacent vertices in a graph are \emph{clones} if they have the same closed neighborhoods.

\begin{lem}\label{lem: modules and twin sets}
If $G$ is a canonically indecomposable split graph, then its maximal proper modules are disjoint, and two vertices belong to the same maximal proper module if and only if they are twins or clones.
\end{lem}
\begin{proof}
Suppose that $G$ is canonically indecomposable and split. Since $G$ is split, it and its complement are both $2K_2$-free (see~\cite{FoldesHammer76}). This implies that if $G$ is disconnected, then all but one of the components are isolated vertices, which contradicts the indecomposability of $G$. A similar contradiction arises if $G$ is the complement of a disconnected graph, so both $G$ and its complement are connected. The modular decomposition theorem of Gallai mentioned previously then implies that the maximal proper modules of $G$ are disjoint.

It is clear that if two vertices are twins or clones, then they belong to the same maximal proper module. By Lemma~\ref{lem: max modules}, these modules lie within the maximum clique $Q$ or within the maximum independent set $I$ of $G$. If two vertices belong to the same maximal proper module $M$, then they have the same open or closed neighborhood, depending on whether $M$ is a subset of $I$ or of $M$, respectively, and hence the vertices are twins or clones.
\end{proof}

We will use Lemmas~\ref{lem: max min degrees} and~\ref{lem: modules and twin sets} multiple times without mention in proving the next lemma.

\begin{lem} \label{lem: module cases}
If $G$ is an indecomposable weakly threshold graph with five or more vertices, then exactly one of the following is true of $G$:
\begin{enumerate}
\item[\textup{(i)}] there is a unique vertex $u$ of maximum degree, and the vertices of minimum degree comprise a maximal proper module with at least two vertices, each of which is adjacent to $u$;
\item[\textup{(ii)}] there is a unique vertex $u$ of maximum degree, and the vertices of minimum degree belong to exactly two distinct maximal proper modules, one of which has size one and contains the unique vertex $v$ not adjacent to $u$;
\item[\textup{(iii)}] there is a unique vertex $v$ of minimum degree, and the vertices of maximum degree comprise a maximal proper module with at least two vertices, each of which is non-adjacent to $v$;
\item[\textup{(iv)}] there is a unique vertex $v$ of minimum degree, and the vertices of maximum degree belong to exactly two distinct maximal proper modules, one of which has size one and contains the unique vertex $u$ adjacent to $v$.
\end{enumerate}
\end{lem}
\begin{proof}
Let $G$ be an indecomposable weakly threshold graph with five or more vertices. By Lemma~\ref{lem: constructing indecomp WT}, $G$ may be constructed from an induced subgraph isomorphic to $P_4$ by iteratively applying a sequence $O_1,\dots,O_p$ of Type 2 operations.  We proceed by induction on the $p$. Observe that if $p\geq 1$, since $G$ has at least five vertices, and if $p=1$, then $G$ is isomorphic to the chair or kite graph, which respectively satisfy cases (ii) and (iv) of the claim.

Assume now that the claim holds for all indecomposable weakly threshold graphs constructed from $P_4$ via $k$ operations of Type 2, where $k$ is some natural number, and suppose that $p=k+1$. Let $w$ be the vertex added during the operation $O_p$. Note that graph complementation preserves modules and the properties of being indecomposable, split, and a weakly threshold graph. Furthermore, under graph complementation weakly dominating vertices become weakly isolated vertices, and vice versa. Thus we may replace $G$ by its complement if desired and assume that the vertex $w$ is a weakly isolated vertex. Let $G' = G-w$. By the induction hypothesis, $G'$ is described by one of the statements (i)--(iv). 

We consider each of those cases in turn. If (i) holds for $G'$, then $G$ has a unique vertex of maximum degree, and $w$ is a twin of vertices having minimum degree in $G'$, creating a larger such module in $G$; hence (i) holds for $G$.

If (ii) holds for $G'$, then $G$ again has a unique vertex $u$ of maximum degree, and $w$ is a twin of the vertices of minimum degree adjacent to $u$, preserving the module of size one containing the vertex not adjacent to $u$; hence (ii) holds for $G$.

If (iii) holds for $G'$, then the addition of $w$ to $G'$ creates exactly two distinct maximal proper modules in $G$, each with just one vertex, that consist of vertices of minimum degree. Furthermore, the neighbor of $w$ is the unique vertex of maximum degree in $G$; hence (ii) holds for $G$.

Finally, if (iv) holds for $G'$, then let $v$ be the vertex of minimum degree in $G'$. In $G$, either $v$ and $w$ are twins, in which case (i) holds for $G$, or $v$ and $w$ have different neighbors, in which case (ii) holds for $G$.
\end{proof}

Recall that for $n \geq 4$, $h_n$ denotes the number of indecomposable weakly threshold graphs on $n$ vertices. Lemma~\ref{lem: module cases} now allows us to derive a recurrence relation for the terms $h_n$.

\begin{thm}\label{thm: b recurrence}
For all $n \geq 7$, we have $h_n = 3h_{n-1} - h_{n-2}$.
\end{thm}
\begin{proof}
For each $n \geq 5$, let $b_n$ denote the number of indecomposable weakly threshold graphs $G$ with $n$ vertices in which the vertices of maximum degree comprise a single module of $G$, and the vertices of minimum degree likewise comprise a single module in $G$. Observe that $b_n$ counts the number of graphs with $n$ vertices that are described in Lemma~\ref{lem: module cases} in cases (i) and (iii). Further define $a_n = h_n - b_n$ for each integer $n \geq 5$.

It follows from Lemma \ref{lem: module cases} that the indecomposable weakly threshold graphs on $n \geq 5$ vertices can each be obtained by adding a weakly dominating vertex or a weakly isolated vertex to an indecomposable weakly threshold graph on $n-1$ vertices. Furthermore, since adding a weakly isolated vertex to a weakly threshold graph creates a graph with at least two vertices of degree 1 and a single vertex of maximum degree, and adding a weakly dominating vertex to a weakly threshold graph creates a graph with at least two vertices of degree $n-2$ and a unique vertex of degree 1, it will never be the case that we can obtain the same indecomposable weakly threshold graph on $n$ vertices by adding a weakly isolated vertex to one weakly threshold graph and adding a weakly dominating vertex to another.

Furthermore, we can determine the number of distinct isomorphism classes that are represented by graphs produced by adding a weakly isolated or weakly dominating vertex to a given indecomposable weakly threshold graph $H$, as we now describe. Our cases come from Lemma~\ref{lem: module cases}.

If $H$ is described by cases (i) or (iii), then up to isomorphism there is one way in which a weakly isolated vertex can be added to $H$, and exactly one way in which a weakly dominating vertex can be added.

If $H$ is described by case (ii), then there is one way to add a weakly isolated vertex, and up to isomorphism there are two ways to add a weakly dominating vertex---we may make the new vertex the clone of an already-existing vertex of maximum degree, or we may make the new weakly dominating vertex non-adjacent to a vertex of minimum degree adjacent to the vertex of maximum degree in $H$. Note that these two ways produce graphs in cases (iii) and (iv), respectively, which hence cannot be isomorphic. For similar reasons, if $H$ is described by case (iv), then there is one way to add a weakly dominating vertex and two ways (up to isomorphism) to add a weakly isolated vertex. 

Thus $h_{n} = 2b_{n-1} + 3a_{n-1}$ for all $n \geq 6$. We now show that $b_n= h_{n-1}$ for all $n \geq 6$. Indeed, note that $b_{n}$ counts the number of $n$-vertex graphs described in cases (i) and (iii). There is a bijection between $n$-vertex graphs satisfying (i) and graphs on $n-1$ vertices satisfying (i) or (iv), given by identifying in the smaller graph a vertex of minimum degree whose neighbor belongs to a maximal proper module of size 1 and then creating a twin for the vertex of minimum degree. There is a similar bijection between $n$-vertex graphs satisfying (iii) and graphs on $n-1$ vertices satisfying (ii) or (iii), completing our proof that $b_n= h_{n-1}$ for all $n \geq 6$.

Now for $n \geq 7$ we can now conclude that \[h_n = 2b_{n-1} + 3a_{n-1} = 2h_{n-2} + 3(h_{n-1}-h_{n-2}) = 3h_{n-1} - h_{n-2}. \qedhere\]
\end{proof}

Using standard techniques, from the recurrence for $h_n$ and the observed values $h_4=1$, $h_5=2$, and $h_6=6$, we can now derive a closed form expression for $H(x)$, obtaining \[H(x) = 2x + \frac{x^4-x^5+x^6}{1-3x+x^2} = \frac{2x-6x^2+2x^3+x^4-x^5+x^6}{1-3x+x^2}.\]

\subsection{General weakly threshold sequences and graphs} \label{subsec: conclusion}

Having obtained expressions for $G(x)$ and $H(x)$, we can now substitute them into equations~\eqref{eq: S(x)} and~\eqref{eq: W(x)} to obtain $S(x)$ and $W(x)$.

\begin{thm} \label{thm: gen fcns}
The generating function for the  weakly threshold sequences, indexed by the number of terms, is \[S(x) = \frac{x-x^2-x^3}{1-3x+x^2+x^3} = -1 + \frac{1}{2(1-x)}+\frac{1-x}{2(1-2x-x^2)}.\] The generating function for the weakly threshold graphs, indexed by the number of vertices, is \[W(x) = \frac{x-2x^2-x^3-x^5}{1-4x+3x^2+x^3+x^5} = -1 + \frac{2}{3(1-x-x^2)} + \frac{1-2x}{3(1-3x+x^2-x^3)}.\]
\end{thm}

Using standard techniques, we obtain formulae for the numbers $s_n$ and $w_n$ of weakly threshold sequences and graphs on $n$ vertices.

\begin{thm} \label{thm: gen fcns}
For integers $n \geq 1$,
\[s_n = \frac{2+\left(1+\sqrt{2}\right)^n + \left(1-\sqrt{2}\right)^n}{4},\]
and
\begin{align*}
w_n = &c_1\left(\frac{1+\sqrt{5}}{2}\right)^n + c_2\left(\frac{1-\sqrt{5}}{2}\right)^n\\
&+ c_3\left(\frac{6-(1+i \sqrt{3}) (27-3 \sqrt{57})^{1/3}-(1-i \sqrt{3}) (27+3\sqrt{57})^{1/3}}{6}\right)^n\\
&+ c_4\left(\frac{6-(1-i \sqrt{3}) (27-3 \sqrt{57})^{1/3}-(1+i \sqrt{3}) (27+3\sqrt{57})^{1/3}}{6}\right)^n\\
&+ c_5\left( \frac{3+(27-3 \sqrt{57})^{1/3}+(27+3\sqrt{57})^{1/3}}{3}\right)^n,
\end{align*}

where

\begin{align*}
c_1 &= \frac{\sqrt{5}+1}{3\sqrt{5}}, \qquad c_2 = \frac{\sqrt{5}-1}{3\sqrt{5}},\\
c_3 &= \frac{1}{9}\left(1 - \frac{(1+i \sqrt{3}) (3 \sqrt{57}-19)^{1/3}}{4\cdot 19^{2/3}}+\frac{1-i \sqrt{3}}{2\cdot (19 (3 \sqrt{57}-19))^{1/3}}\right),\\
c_4 &= \frac{1}{9}\left(1 - \frac{(1-i \sqrt{3}) (3 \sqrt{57}-19)^{1/3}}{4\cdot 19^{2/3}}+\frac{1+i \sqrt{3}}{2\cdot (19 (3 \sqrt{57}-19))^{1/3}}\right), \text{ and}\\
c_5 &= \frac{1}{9}\left(1+\frac{(3 \sqrt{57}-19)^{1/3}}{2 \cdot 19^{2/3}}- \frac{1}{(19 (3 \sqrt{57}-19))^{1/3}}\right).
\end{align*}
\end{thm}

We close this section with a few remarks. First, the sequence $(s_n)$ has previously appeared in the Online Encyclopedia of Integer Sequences, matching sequences A171842 and A024537~\cite{OEIS}. The sequence is reported to count several other sets of combinatorial objects; it may be interesting to find correspondences between weakly threshold sequences and these objects.

Next, we recall that the number of threshold graphs on $n$ vertices is precisely $2^{n-1}$, and that threshold graphs are the unique realizations of their degree sequences. In comparison, Theorem \ref{thm: gen fcns} indicates that \begin{equation}
s_n \;\sim\; \frac{1}{4}(1+\sqrt{2})^n \qquad \text{ and } \qquad w_n \;\sim\; c_5\left( \frac{3+(27-3 \sqrt{57})^{1/3}+(27+3\sqrt{57})^{1/3}}{3}\right)^n.
\label{eq: asymptotics}
\end{equation} That there are more weakly threshold graphs of a given order than weakly threshold sequences is not surprising, since adding a weakly isolated vertex (or weakly dominating vertex) to a weakly threshold graph can often involve a choice of the neighborhood of the added vertex, though the resulting degree sequence is the same no matter which choice is made. Approximating in \eqref{eq: asymptotics}, we see that for large $n$, $s_n \geq \frac{1}{4} \cdot 2.4^n$ and $w_n \geq c_5 \cdot 2.7^n$.

\end{document}